\documentclass{amsart}

\usepackage{amsmath,amssymb,amsthm,amscd}

\usepackage{pinlabel}

\newtheorem{prop}{Proposition}[section]
\newtheorem{thm}[prop]{Theorem}
\newtheorem{lem}[prop]{Lemma}

\theoremstyle{definition}

\newtheorem{ex}[prop]{Example}
\newtheorem{rem}[prop]{Remark}

\newtheorem*{ack}{Acknowledgements}

\def\co{\colon\thinspace}

\newcommand{\C}{\mathbb{C}}

\newcommand{\rme}{\mathrm{e}}

\newcommand{\rmi}{\mathrm{i}}

\newcommand{\tlambda}{\tilde{\lambda}}

\newcommand{\tM}{\widetilde{M}}
\newcommand{\tmu}{\tilde{\mu}}

\newcommand{\N}{\mathbb{N}}

\newcommand{\R}{\mathbb{R}}
\newcommand{\RP}{\mathbb{R}\mathrm{P}}

\newcommand{\Z}{\mathbb{Z}}

\DeclareMathOperator{\Int}{Int}

\begin{document}

\author[H.~Geiges]{Hansj\"org Geiges}

\author[N.~Thies]{Norman Thies}
\address{Mathematisches Institut, Universit\"at zu K\"oln,
Weyertal 86--90, 50931 K\"oln, Germany}
\email{geiges@math.uni-koeln.de, nthies1@smail.uni-koeln.de}

\title{Klein bottles in lens spaces}

\date{}

\begin{abstract}
Bredon and Wood have given a complete answer to the embeddability
question for nonorientable surfaces in lens spaces. They formulate
their result in terms of a recursive formula that determines,
for a given lens space, the minimal genus of embeddable
nonorientable surfaces. Here we give a direct proof that,
amongst lens spaces as target manifolds, the Klein bottle embeds
into $L(4n,2n\pm 1)$ only. We describe four explicit realisations
of these embeddings.
\end{abstract}



\maketitle


\section{Introduction}
A nonorientable closed manifold never embeds as a hypersurface into
euclidean space (or a sphere). A strikingly beautiful proof of this fact
is due to Samelson~\cite{same69}. Although it is hardly possible
to be more concise than Samelson, here is a pr\'ecis of his proof.

Suppose $M\subset\R^n$ were a nonorientable hypersurface.
Take a simple closed path in $M$ along which the orientation
is reversed; then the same must be true for the coorientation.
By pushing the path in the normal direction, and then joining
initial and end point by a short line segment transverse to~$M$,
one obtains an embedded loop in $\R^n$ intersecting $M$
in a single point. The embedding of this loop extends to a smooth
mapping $D^2\rightarrow\R^n$ of a disc, which can be made
(rel boundary) transverse to~$M$. But then the preimage of $M$
in $D^2$ is a $1$-dimensional closed submanifold with
a single boundary point --- contradiction. (Notice that $M$ need not
be a closed manifold for us to reach this contradiction; it suffices
that $M$ be a submanifold without boundary that is a closed
subset of~$\R^n$.)

Thus, it is a natural question to ask about embeddings of
nonorientable surfaces into the `most simple' orientable $3$-manifolds
beyond the $3$-sphere. We take these to be $3$-manifolds with
a Heegaard splitting of genus~$1$, that is, lens spaces and $S^1\times S^2$.

\begin{rem}
By `lens spaces' we shall mean `honest' lens spaces $L(p,q)$ with
$p\geq 1$, which are quotients of $S^3$ under a linear $\Z_p$-action.
In other words, we do not regard $L(0,1)=S^1\times S^2$ as a lens space,
but instead discuss it separately.
\end{rem}

Explicitly, $L(p,q)$ is defined, as an oriented manifold,
as the quotient of $S^3\subset\C^2$
under the $\Z_p$-action generated by
\begin{equation}
\label{eqn:sigma}
\sigma(z_1,z_2)=\bigl(\rme^{2\pi\rmi/p}z_1,\rme^{2\pi\rmi q/p}z_2\bigr).
\end{equation}
Here $p\in\N$ and $q\in\Z$ are assumed to be coprime. Without loss of
generality one may require that $1\leq q\leq p-1$.

The following theorem is Corollary 6.4 in \cite{brwo69}.
See also \cite{end92} for alternative proofs of the
results of Bredon and Wood.

\begin{thm}
\label{thm:BW}
Amongst lens spaces as target manifolds,
the Klein bottle embeds precisely into
$L(4n,2n\pm 1)$, $n\in \N$.
\end{thm}

Bredon and Wood derive this statement as a corollary of their
much more general results about the minimal genus of a nonorientable
surface that can be embedded into a given lens space. (It is clear that one
can always increase the genus by $2$; simply form the connected sum with
a small $2$-torus embedded in a ball inside the ambient $3$-manifold.)
They present a recursive formula that allows one to compute this minimal
genus.

The generality of their argument
slightly obscures the beautiful geometry behind the
embeddings of Klein bottles. Our aim in this note, therefore,
is to present a direct and elementary proof of
Theorem~\ref{thm:BW}. Furthermore, we describe four explicit
realisations of these embeddings, some of which do not seem to
have appeared in the literature before.
\section{Embeddings of the projective plane}
As a warm-up, we discuss embeddings of the projective plane $\RP^2$
into lens spaces. Obviously, $\RP^2$ embeds into $\RP^3=L(2,1)$ by
inclusion.

\begin{prop}
The only lens space into which $\RP^2$ embeds is $L(2,1)$.
\end{prop}

\begin{proof}
Let $\RP^2\subset L(p,q)$ be an embedded copy of the
projective plane. Since $L(p,q)$
is orientable, and $\RP^2$ is not, a (closed) tubular neighbourhood
$\nu\RP^2$ is diffeomorphic to the total space
of a nontrivial $I$-bundle over $\RP^2$, where $I$
denotes a compact interval.

There is a unique such bundle, which can be seen by regarding
$\RP^2$ as being obtained by gluing a $2$-disc $D^2$ and a M\"obius
band along their respective boundary circle. Over the
M\"obius band, the $I$-bundle must be nontrivial, and
since the M\"obius band retracts to a circle, this bundle
is unique. The boundary circle of the M\"obius band is a double
cover of its spine, so there the $I$-bundle is trivial,
as it is over~$D^2$. The gluing of these two bundles
is unique, since there are no nontrivial loops in
the diffeomorphism group of~$I$.

We claim that $\partial(\nu\RP^2)$ is a $2$-sphere.
To this end, think of $\RP^3$ as a $3$-ball $D^3$ with antipodal
points on $S^2=\partial D^3$ identified. Since $S^2$ descends
to an embedded $\RP^2$ in the quotient, we see that the total space
of the unique nontrivial $I$-bundle over $\RP^2$ is diffeomorphic
to $\RP^3$ with a small open $3$-ball $B^3_{\varepsilon}$ removed.

Now, every lens space is irreducible~\cite[Proposition~1.6]{hatc07}
--- that is, any embedded $2$-sphere bounds a $3$-ball ---
since it is covered by the $3$-sphere, which is irreducible by
Alexander's theorem~\cite[Theorem~1.1]{hatc07}.

Since $\nu\RP^2\simeq\RP^2$ is not a ball, we deduce that the ambient
space $L(p,q)$ is obtained from
$\nu\RP^2=\RP^3\setminus B^3_{\varepsilon}$ by attaching
a copy of~$D^3$. The result of this gluing is unique: up to homeomorphism
by the Alexander trick (any homeomorphism of
$S^2$ extends to a homeomorphism of~$D^3$), up to diffeomorphism by
Smale's theorem~\cite{smal59}. We conclude that $L(p,q)=L(2,1)$.
\end{proof}

\begin{rem}
Observe that our argument has shown more than is stated
in the proposition. The only information we had to use
about the ambient manifold was that any \emph{separating}
$2$-sphere bounds a $3$-ball, in other words, that the
ambient manifold is prime. Thus, $\RP^3$ is the only orientable
prime $3$-manifold into which $\RP^2$ can be embedded. In particular,
$\RP^2$ does not embed into $S^1\times S^2$ either.
\end{rem}
\section{Heegaard splitting of $L(p,q)$}
The following result, which is standard textbook material,
is crucial for much of our discussion, so we give a brief
indication how to prove this statement.

\begin{prop}
\label{prop:lens-heegaard}
Choose $r,s\in\Z$ such that $pr+qs=1$. Then $L(p,q)$ has a Heegaard splitting
of genus~$1$, where the gluing of the two solid tori is given by
\begin{equation}
\label{eqn:matrix}
\begin{array}{rcl}
\mu_1 & \sim & p\lambda_2-q\mu_2,\\
\lambda_1 & \sim & s\lambda_2 + r\mu_2.
\end{array}
\end{equation}
\end{prop}

Observe that with respect to the bases $(\mu_i,\lambda_i)$, the gluing map
is described by a matrix
of determinant~$-1$, that is, the gluing map reverses the orientation.

\begin{proof}[Proof of Proposition~\ref{prop:lens-heegaard}]
It is convenient to think of $S^3\subset\C^2$ as the $3$-sphere of
radius $\sqrt{2}$, that is, $S^3=\{|z_1|^2+|z_2|^2=2\}$.
A genus $1$ Heegaard splitting of $S^3\subset\C^2$ is then given by
the solid tori $\tM_i:=\{(z_1,z_2)\in S^3\co z_i\leq 1\}$, $i=1,2$.
A diffeomorphism from $S^1\times D^2$ to $\tM_2$ is defined by
\[ \bigl(\rme^{\rmi\theta},r\rme^{\rmi\varphi}\bigr)\longmapsto
\bigl(\sqrt{2-r^2}\,\rme^{\rmi\theta},r\rme^{\rmi\varphi}\bigr).\]
In terms of these coordinates, the generator $\sigma$ of the $\Z_p$-action
as in \eqref{eqn:sigma} acts by
\[ \sigma|_{\tM_2}\co\theta\longmapsto \theta+\frac{2\pi}{p},\;\;\;
\varphi\longmapsto\varphi+\frac{2\pi q}{p}.\]
A fundamental domain for this action is $\{0\leq\theta\leq 2\pi/p\}$,
and the quotient $M_2:=\tM_2/\langle\sigma\rangle$ is again a solid
torus.

The meridian $\tmu_2$ of $\partial\tM_2$ is given by a $\varphi$-curve,
and this descends to a meridian $\mu_2$ of $\partial M_2$. As longitude
$\tlambda_2$ on $\partial\tM_2$ we take a $\theta$-curve, and as longitude
$\lambda_2$ on $\partial M_2$ a curve joining $(\theta_0,\varphi_0):=(0,0)$
with $(\theta_1,\varphi_1):=(2\pi/p,2\pi q/p)$ as shown
in Figure~\ref{figure:lens-heegaard}. Here it is assumed that
$\lambda_2$ does indeed turn through an angle $2\pi q/p$ in meridional
direction, and we need not require $1\leq q\leq p-1$.

The horizontal line segments shown in the figure are two of the $p$
segments that make up $\tlambda_2$, after they have all been mapped to the
fundamental domain. We then see that $\tlambda_2$ descends
to $p\lambda_2-q\mu_2$.

In a similar way one considers $M_1=\tM_1/\langle\sigma\rangle$.
In the parametrisation of $\tM_1$ as a solid torus,
\[ \bigl(\rme^{\rmi\theta},r\rme^{\rmi\varphi}\bigr)\longmapsto
\bigl(r\rme^{\rmi\varphi},\sqrt{2-r^2}\,\rme^{\rmi\theta}\bigr),\]
the roles of $z_1$ and $z_2$ are interchanged, so $\sigma$ acts by
\[ \sigma|_{\tM_1}\co\theta\longmapsto \theta+\frac{2\pi q}{p},\;\;\;
\varphi\longmapsto\varphi+\frac{2\pi}{p}.\]
(Beware that, as before, $\varphi$ is the meridional coordinate, and
$\theta$ the longitudinal one.)
If one replaces the generator $\sigma$ of the $\Z_p$-action by $\sigma^s$,
the action is described by
\[ \sigma^s|_{\tM_1}\co\theta\longmapsto \theta+\frac{2\pi}{p},\;\;\;
\varphi\longmapsto\varphi+\frac{2\pi s}{p}.\]
Now the discussion continues as before, and we see that there is a choice
of longitude $\lambda_1$ on $M_1$ such that $\tmu_1$ descends to $\mu_1$,
and $\tlambda_1$ to $p\lambda_1-s\mu_1$.

\begin{figure}[h]
\labellist
\small\hair 2pt
\pinlabel $\mu_2$ [br] at 230 198
\pinlabel $\lambda_2$ [bl] at 356 101
\pinlabel $\tlambda_2$ [t] at 365 59
\pinlabel $\tlambda_2$ [b] at 108 168
\pinlabel $\theta=0$ [b] at 34 225
\pinlabel $\theta=2\pi/p$ [b] at 467 225
\pinlabel $\varphi=0$ [r] at 0 166
\pinlabel $\varphi=2\pi q/p$ [r] at 0 59
\endlabellist
\centering
\includegraphics[scale=0.4]{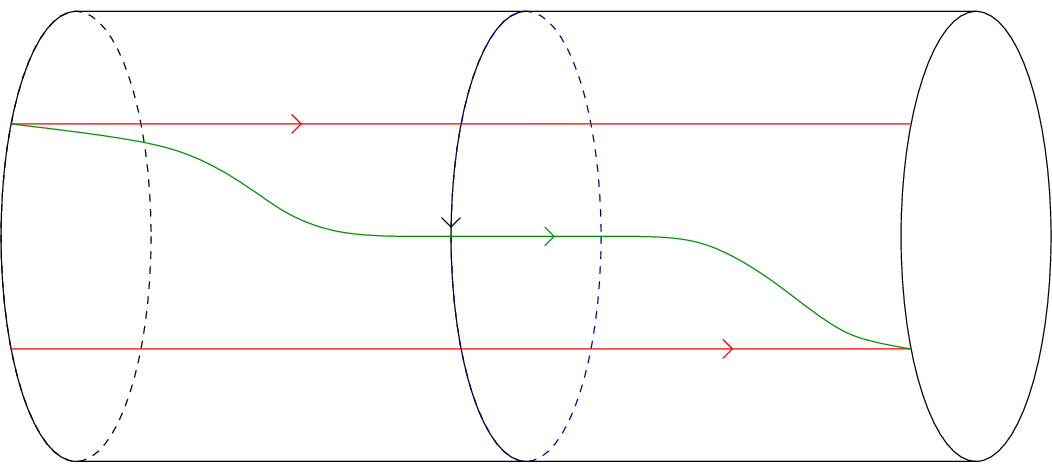}
  \caption{A Heegaard (solid) torus of $L(p,q)$.}
  \label{figure:lens-heegaard}
\end{figure}
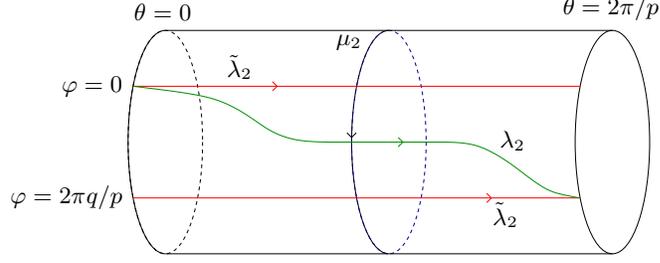

In the Heegaard splitting of~$S^3$, the identification of the two solid
tori is determined by
\[ \tmu_1\sim\tlambda_2\;\;\;\text{and}\;\;\;
\tlambda_1\sim\tmu_2.\]
This descends to
\[ \mu_1\sim p\lambda_2-q\mu_2\;\;\;\text{and}\;\;\;
p\lambda_1-s\mu_1\sim\mu_2,\]
which is equivalent to the identifications in the proposition.
\end{proof}
\section{Klein bottles embed into $L(4n,2n\pm 1)$ only}
\label{section:Lonly}
Suppose we have an embedded copy $K\subset L(p,q)$ of the
Klein bottle. Think of $K$ as being obtained by gluing
two M\"obius bands along their boundary. As in the preceding section,
we see that the normal bundle of $K$ in $L(p,q)$ has to be
the unique nontrivial $I$-bundle over either of these M\"obius bands,
where we take $I$ to be the interval $[-1,1]$.
The boundary (in the fibre direction) of this $I$-bundle
is the orientable double cover of the M\"obius band, that is,
an annulus. It follows that the boundary of a tubular neighbourhood
$\nu K$ is a separating $2$-torus $T$, the orientable double
covering of~$K$.

\begin{rem}
\label{rem:pi1}
We shall see presently that the homomorphism $\pi_1(T)\rightarrow
\pi_1(\nu K)$ induced by inclusion is given by
\[ \begin{array}{ccc}
\langle a,b\,|\,ab=ba\rangle & \longrightarrow & \langle u,v\,|\,
                                                 uvu^{-1}=v^{-1}\rangle\\
a^nb^{\ell}                  & \longmapsto     & u^{2n}v^{\ell}.
\end{array} \]
This homomorphism is injective, for otherwise there would be
some relation $u^{2n}v^{\ell}=1$ in $\pi_1(K)$. However, adding such a
relation to the given presentation of $\pi_1(K)$ turns
$u$ into a torsion element of order $2n$ for $\ell=0$, or
$v$ into a torsion element of order~$2\ell$ for $\ell\neq 0$ (since
$uv^{\ell}u^{-1}=v^{-\ell}$ by the first relation, and
$uv^{\ell}u^{-1}=v^{\ell}$ by the second). This would contradict the fact
that $\pi_1(K)$ is torsion-free, as indeed is the
fundamental group of any finite-dimensional $CW$ complex with
a contractible universal covering space \cite[Proposition 2.45]{hatc02}.
\end{rem}
\subsection{The complement of $\nu K$}
We first show that the complement of $\nu K$ in $L(p,q)$ is
a solid torus. This observation has been recorded in various places in the
literature, e.g.~\cite[Lemma~2]{rubi79}.

\begin{lem}
The complement $L(p,q)\setminus\Int(\nu K)$ is a solid torus. In other words,
$L(p,q)$ is a Dehn filling of $\nu K$.
\end{lem}

\begin{proof}
The homomorphism $\pi_1(T)\rightarrow\pi_1\bigl(L(p,q)\bigr)$
on fundamental groups induced by the inclusion $T\subset L(p,q)$
cannot be injective, so $T$ is compressible \cite[Corollary~3.3]{hatc07},
that is, there is an embedded disc $D$ in $L(p,q)$ with
$D\cap T=\partial D$ a homotopically nontrivial circle on~$T$.
By Remark~\ref{rem:pi1}, this disc has to lie on the other
side of $T$ than $\nu K$.

Now thicken $D\equiv D\times\{0\}$
to a cylinder $Z:=D\times[-\varepsilon,\varepsilon]$
with $Z\cap T=\partial D\times[-\varepsilon,\varepsilon]$, and replace
this part of $T$ by $D\times\{\pm\varepsilon\}$, thus creating a
$2$-sphere~$S$. This process is called a surgery of $T$ along~$D$.

In an irreducible $3$-manifold (such as a lens space),
this $2$-sphere $S$ bounds a ball~$B$. If $B\cap D=\emptyset$, reversing the
surgery creates a solid torus bounded by~$T$ on the other
side of $T$ than $\nu K$, and we are done.

If $B$ were on the side of $S$ containing $D$, then $B$ would equal
\[ \nu K\cup _{\partial D\times[-\varepsilon,\varepsilon]}
\bigl(D\times [-\varepsilon,\varepsilon]\bigr).\]
In particular, this ball would contain a Klein bottle, which is impossible
by Samelson's theorem.
\end{proof}
\subsection{The fundamental group of the Dehn filling}
Next, we wish to describe an explicit model for
$\nu K$ that will allow us to compute the fundamental group
of the manifold obtained by any Dehn filling of $\nu K$.

We think of $S^1$ as $\R/2\pi\Z$ and write the Klein bottle as
\[ K= \bigl([0,1]\times S^1\bigr)/(1,\theta)\sim(0,-\theta).\]
The tubular neighbourhood $\nu K$ of $K$ in any orientable
$3$-manifold is then given by
\[ \nu K=\bigl([0,1]\times S^1\times[-1,1]\bigr)/
(1,\theta,r)\sim (0,-\theta,-r),\]
where $K\subset\nu K$ is defined by $\{r=0\}$.
Indeed, notice that $K$ decomposes into the two
M\"obius bands
\[ \Bigl([0,1]\times\Bigl[-\frac{\pi}{2},\frac{\pi}{2}\Bigr]
\times\{0\}\Bigr)/\!\sim\]
and
\[ \Bigl([0,1]\times\Bigl[\frac{\pi}{2},\frac{3\pi}{2}\Bigr]
\times\{0\}\Bigr)/\!\sim,\]
and over either M\"obius band the $[-1,1]$-bundle
is nontrivial.

The first homology group of $\nu K$ is $H_1(\nu K)\cong
\Z\oplus\Z_2$, where the $\Z$-summand may be taken to be generated
by $\bigl([0,1]\times\{0\}\times\{0\}\bigr)/\!\sim$,
and the $\Z_2$-summand is generated by $\bigl\{\frac{1}{2}\bigr\}
\times S^1\times\{0\}$. Notice that the latter circle
is isotopic, by sliding it into the $[0,1]$-direction, to
a copy of itself with reversed orientation, so two copies
of this circle bound a cylinder.

Now consider the effect of attaching a solid torus $S^1\times D^2$
to~$\nu K$. To compute the first homology of the
resulting space, it suffices to consider the attaching of a
meridional disc $D:=\{*\}\times D^2$, for the attaching of
a solid torus may be thought of as an attaching of a disc,
followed by the attaching of a $3$-ball.

As remarked before, the boundary of $\nu K$ is a $2$-torus~$T$.
In our model, this is
\[ T=\partial(\nu K)=
\bigl([0,1]\times S^1\times\{\pm 1\}\bigr)/\!\sim.\]
We take $H_1(T)\cong\Z\oplus\Z$ to be generated by
$\bigl([0,1]\times\{0\}\times\{\pm 1\}\bigr)/\!\sim$
and $\bigl\{\frac{1}{2}\bigr\}\times S^1\times\{1\}$.
Then the homomorphism on homology induced
by the inclusion $T\rightarrow\nu K$ is described by
\[ \begin{array}{ccc}
H_1(T)     & \longrightarrow & H_1(\nu K)\\
\Z\oplus\Z & \longrightarrow & \Z\oplus\Z_2\\
(n,\ell)   & \longmapsto     & (2n,\ell\,\mathrm{mod}\,2).
\end{array}\]

The attaching circle $\partial D$ of the meridional disc
represents an element $(n,\ell)\in H_1(T)$, with $n,\ell$ coprime.
Thus, assuming that the corresponding gluing results
in a lens space $L=L(p,q)$, the Mayer--Vietoris sequence becomes
\[ \begin{array}{ccccccc}
H_1(\partial D) & \longrightarrow & H_1(\nu K)
   & \longrightarrow & H_1(L) & \longrightarrow & 0\\
\Z                & \longrightarrow & \Z\oplus\Z_2
   & \longrightarrow & \Z_p   & \longrightarrow & 0\\
1                 & \longmapsto     & (2n,\ell\;\mathrm{mod}\, 2)
   &                 &        &                 & 
\end{array}\]
It follows that $\ell$ must be odd, for the quotient
$(\Z\oplus\Z_2)/\langle(2n,0)\rangle$ equals $\Z_{2n}\oplus\Z_2$,
and not $\Z_p$, as it should. On the other hand, we have
\[ (\Z\oplus\Z_2)/\langle(2n,1)\rangle\cong Z_{4n},\]
since $\Z\oplus\Z_2$ is generated by $(2n,1)$ and $(1,0)$,
and we have $(4n,0)=2(2n,1)$. So with this attaching map
the resulting manifold might indeed be a lens space,
with $p=4n$.

To determine the freedom in choosing~$\ell$, we
compute the fundamental group in a similar fashion.
Let $u,v\in\pi_1(\nu K)$ be the elements corresponding to the very
circles we chose as generators of $H_1(\nu K)$. These yield the
presentation
\[ \pi_1(\nu K)=\langle u,v\,|\,uvu^{-1}=v^{-1}\rangle,\]
and in terms of this presentation, the homomorphism
$\pi_1(T)\rightarrow\pi_1(\nu K)$ induced by inclusion
is as claimed in Remark~\ref{rem:pi1}.
Notice that $u^2v=uv^{-1}u=vu^2$, that is, $u^2$ commutes
with~$v$. So the map in Remark~\ref{rem:pi1}
is indeed a group homomorphism.

Thus, $u^{2n}v^{\ell}$ is the class of the
attaching circle $\partial D$ in $\pi_1(\nu K)$.
By Seifert--van-Kampen we then have
\[ \pi_1(\nu K\cup D)=\langle u,v\,|\,
uvu^{-1}=v^{-1},\, u^{2n}v^{\ell}=1\rangle.\]
Since $\ell\neq 0$, this is a presentation of the metacyclic group
\[ \Z_{2\ell}\rightarrowtail G
\twoheadrightarrow\Z_{2n},\]
which is abelian (and then cyclic of order~$4n$)
if and only if $\ell=\pm 1$, see \cite{gela21}.
The normal subgroup of order $2\ell$ is generated by~$v$. Every element
of $G$ can be written uniquely in the form $u^jv^k$
with $0\leq j\leq 2n-1$ and $0\leq k\leq 2\ell-1$, and the projection
to $\Z_{2n}$ is given by $u^jv^k\mapsto j$.

To summarise, we have shown that the only Dehn fillings of $\nu K$
that might result in a lens space $L(p,q)$ are those described by
an attaching map $(n,\pm 1)\in\Z\oplus\Z=H_1(T)$, and that the
resulting lens space would satisfy $p=4n$.
\subsection{Seifert fibrations}
We now show that the manifolds obtained as a Dehn filling of $\nu K$
described by an attaching map $(n,\pm 1)$ are Seifert fibred in two
different ways. Either fibration will subsequently allow us
to show that the Dehn filling is indeed a lens space
$L(4n,2n\pm 1)$.

Observe that the map $(z_1,z_2)\mapsto
(z_1,\overline{z}_2)$ induces an orientation-reversing
diffeomorphism from $L(4n,2n+1)$ to $L(4n,2n-1)$.
Likewise, there is an orientation-reversing diffeomorphism
of the Dehn fillings corresponding to $(n,1)$ and $(n,-1)$,
respectively, extending the diffeomorphism
$[(t,\theta,r)]\mapsto[(t,-\theta,r)]$ of~$\nu K$.
So it suffices to consider the attaching map
$(n,1)$ only, and we may ignore questions of orientation.
\subsubsection{A Seifert fibration over $\RP^2(n)$}
\label{subsubsection:seifert-rp2}
The first Seifert fibration will yield a very simple description
of an embedded Klein bottle, see Section~\ref{subsection:emb-Seifert}.

\begin{lem}
\label{lem:seifert-rp2}
Given an attaching map described by $(n,1)\in
\Z\oplus\Z=H_1(T)$, the resulting Dehn filling of $\nu K$ has a
Seifert fibration over $\RP^2$ with
one singular fibre of order~$n\in\N$.
\end{lem}

\begin{proof}
The tubular neighbourhood $\nu K$ equals the total space
of the nontrivial $S^1$-bundle over the M\"obius band;
the bundle projection is induced by the map
\[ \begin{array}{ccc}
[0,1]\times S^1\times [-1,1] & \longrightarrow &
[0,1]\times [-1,1] \\
(t,\theta,r)                 & \longmapsto     &
(t,r),
\end{array} \]
which descends to the quotients under the identification
$(1,\theta,r)\sim(0,-\theta,-r)$ and
$(1,r)\sim (0,-r)$, respectively. With
our choice of generators for $H_1(T)$, the class of the $S^1$-fibre
$\bigl\{\frac{1}{2}\bigr\}\times S^1\times\{1\}$
in $\partial(\nu K)=T$ is $(0,1)\in\Z\oplus\Z$. Notice that the restriction
of the $S^1$-bundle to the boundary of the M\"obius band is the
trivial bundle $T\rightarrow[0,1]\times\{\pm 1\}/\!\sim$.

We now look at the gluing map $\partial(S^1\times D^2)
\rightarrow T$. By assumption, the
meridian $\mu:=\{*\}\times \partial D^2\subset
\partial (S^1\times D^2)$ is mapped
to a curve in $T$ representing the class
$(n,1)\in H_1(T)$.
The longitude $\lambda:=S^1\times\{*\}\subset
\partial (S^1\times D^2)$ of the solid torus we glue in
has to map to a curve that forms a basis of
$H_1(T)$ together with the image of~$\mu$;
the most simple choice is $\lambda\sim (1,0)$.

With these choices, the fibre class $(0,1)$
is identified with $\mu-n\lambda$. The foliation
of $\partial(S^1\times D^2)$ by simple closed curves
in this class $\mu-n\lambda$ extends radially in the obvious
fashion to a Seifert fibration of
$S^1\times D^2$ over $D^2$ with one singular fibre
(the spine $S^1\times\{0\}$) of order~$n$.
\end{proof}

In \cite{gela21} it was shown that these Seifert fibred manifolds
are precisely the lens spaces $L(4n,2n\pm 1)$.
\subsubsection{A Seifert fibration over $S^2(2,2)$}
The second Seifert fibration will be shown to
translate into a genus $1$ Heegaard splitting
of the manifold, from which one can read off directly that
the Dehn filling is $L(4n,2n\pm 1)$.

\begin{lem}
The Dehn filling of $\nu K$ resulting from an attaching map described
by $(n,1)\in\Z\oplus\Z=H_1(T)$ has a Seifert fibration over $S^2$ with
two singular fibres of order~$2$.
\end{lem}

\begin{proof}
In the model
\[ \nu K=\bigl([0,1]\times S^1\times[-1,1]\bigr)/
(1,\theta,r)\sim (0,-\theta,-r),\]
the pairs of intervals $[0,1]\times\{\theta\}\times\{r\}$
and $[0,1]\times\{-\theta\}\times\{-r\}$, for $(\theta,r)\not\in
\bigl\{(0,0),(\pi,0)\bigr\}$, define circles of length~$2$.
The two exceptional intervals $[0,1]\times\{0\}\times\{0\}$
and $[0,1]\times\{\pi\}\times\{0\}$ define circles of length~$1$. This
foliation by circles defines a Seifert fibration of $\nu K$ with quotient
\[ S^1\times[-1,1]/(\theta,r)\sim(-\theta, -r),\]
which equals $D^2(2,2)$, the disc with two orbifold points of order~$2$,
see Figure~\ref{figure:D22}.

\begin{figure}[h]
\labellist
\small\hair 2pt
\pinlabel $/\!\sim$ [b] at 248 109
\pinlabel $\cong$ [b] at 542 109
\endlabellist
\centering
\includegraphics[scale=0.4]{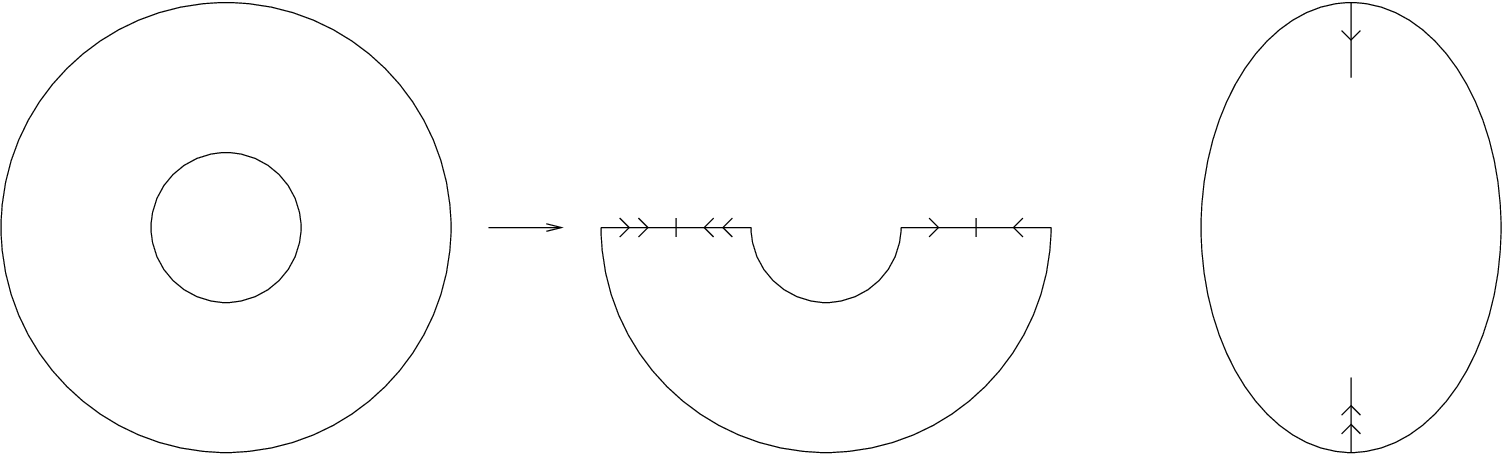}
  \caption{The base orbifold $D(2,2)$.}
  \label{figure:D22}
\end{figure}

On the boundary $T$ of $\nu K$, the Seifert fibres lie in the
class $(1,0)\in\Z\oplus\Z=H_1(T)$. As seen in the proof of
Lemma~\ref{lem:seifert-rp2}, this class becomes identified with
the longitude $\lambda$ of the solid torus producing the Dehn filling,
so the Seifert fibration extends as the  product fibration of
$S^1\times D^2$. Therefore, the Dehn filling results
in a Seifert fibration over $S^2(2,2)$.
\end{proof}

\begin{rem}
More generally, it is shown in \cite{rubi79} that
the irreducible $3$-manifolds with finite fundamental
group that contain Klein bottles
are precisely the Seifert fibrations over $S^2$
with at most three singular fibres of multiplicity $2,2,p$, respectively,
for some $p\in\N$.
\end{rem}

The Seifert fibration, restricted to two hemispheres $D^2(2)$ of
$S^2(2,2)$, each containing one of the two orbifold points, defines
a Heegaard splitting of genus~$1$. This shows that the Dehn filling
produces a lens space, which we now want to identify.

\begin{prop}
The Dehn filling of $\nu K$ with the meridian of the solid torus
glued to $(n,1)\in\Z\oplus\Z=H_1(T)$ is the lens space $L(4n,2n\pm 1)$.
\end{prop}

\begin{proof}
In $\nu K$ we consider tubular neighbourhoods $V_1,V_2$ of the
two singular fibres over the orbifold points $(0,0)$ and $(\pi,0)$
in $S^1\times[-1,1]/\!\sim$. In Figure~\ref{figure:orbifold-heegaard}
these tubular neighbourhoods are represented as holes in a slice
$\{t_0\}\times S^1\times[-1,1]$, invariant under the action
$(\theta,r)\mapsto (-\theta,-r)$ on $S^1\times[-1,1]$.
Here the outer boundary of the annulus is given by $\{r=1\}$,
the inner boundary, by $\{r=-1\}$, and the angular coordinate
$\theta$ is the usual one in the euclidean plane.

\begin{figure}[h]
\labellist
\small\hair 2pt
\pinlabel $\mu_1$ [bl] at 409 280
\pinlabel $\mu_1'$ [br] at 352 280
\endlabellist
\centering
\includegraphics[scale=0.35]{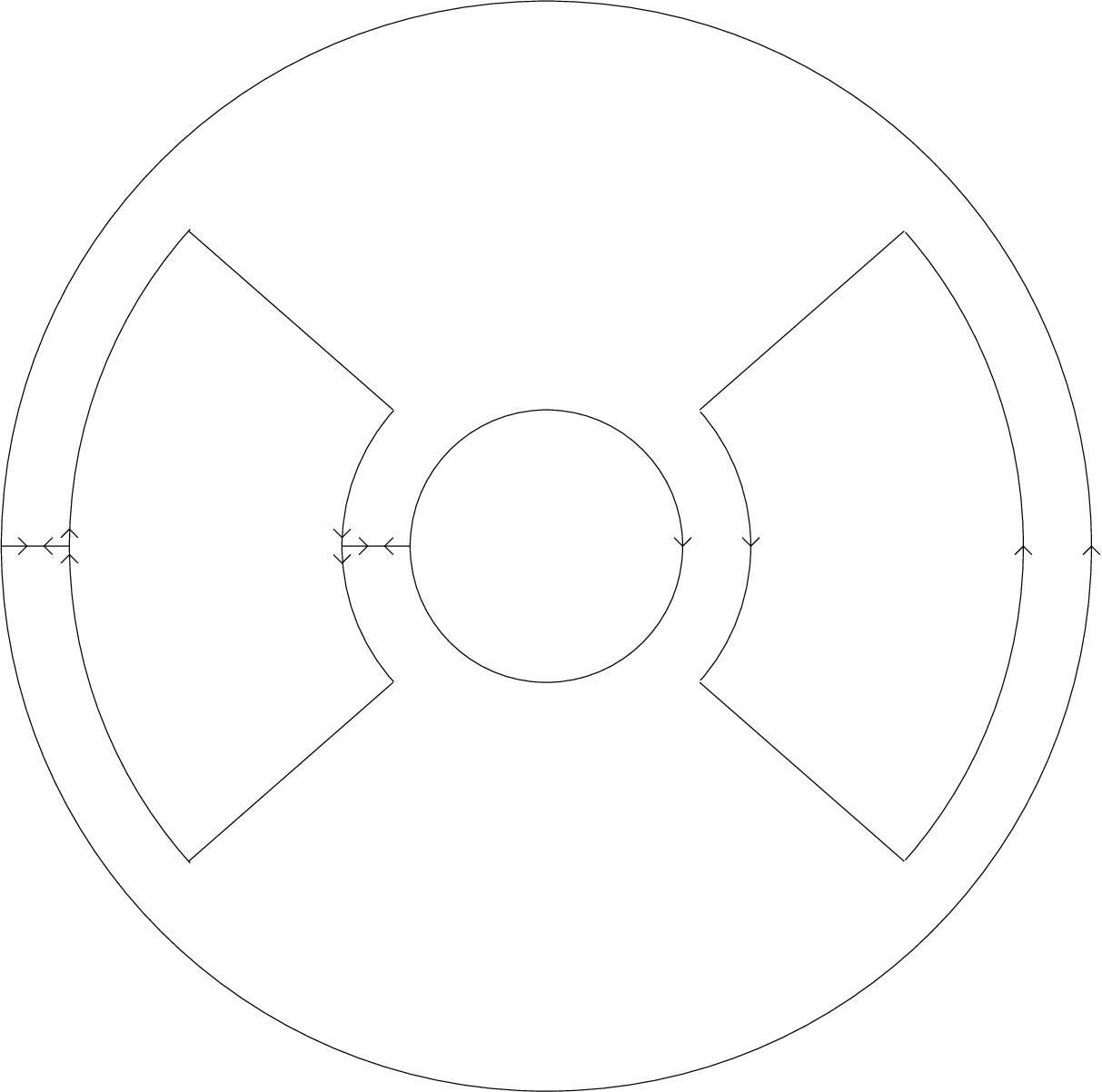}
  \caption{The neighbourhoods of the two singular fibres.}
  \label{figure:orbifold-heegaard}
\end{figure}

The complement of $V_1\cup V_2$ in $\nu K$, together with the solid
torus $S^1\times D^2$ producing the Dehn filling, is a thickened $2$-torus.
In order to identify the lens space resulting from the
Heegaard splitting, it suffices to consider a meridian $\mu_1$ on
$\partial V_1$ (in Figure~\ref{figure:orbifold-heegaard} on the
right), and isotope it in this thickened $2$-torus to a curve
on~$\partial V_2$. Notice that we may think homologically, since
two simple closed curves on a surface are isotopic if and
only if they are homotopic (Baer's theorem); and on a $2$-torus,
with abelian fundamental group, homotopy of curves equals homological
equivalence.

In $\nu K\setminus (V_1\cup V_2)$, the meridian $\mu_1$
is isotopic to the curve $\mu_1'$ as shown. Notice that the
two oriented circles in $\mu_1'$ on the boundary $T$ of $\nu K$
correspond to the same class $(0,1)\in H_1(T)$. On $\partial(S^1\times D^2)$,
each of these curves becomes identified with $\mu-n\lambda$,
see the proof of Lemma~\ref{lem:seifert-rp2}. Since we are now allowed
to isotope the curves over $S^1\times D^2$, these are isotopic to two copies
of $-n\lambda$. The class $-2n\lambda$, in turn, is identified
with $(-2n,0)\in H_1(T)$.

Next we need to recall that the class $(1,0)\in H_1(T)$ is
represented by a regular Seifert fibre. In the
cylinder over the disc with three holes shown in
Figure~\ref{figure:orbifold-heegaard}, this corresponds to two
intervals $[0,1]\times\{\theta_0\}\times\{r_0\}$
and $[0,1]\times\{-\theta_0\}\times\{-r_0\}$.

In $\nu K\setminus (V_1\cup V_2)$,
any two such Seifert fibres are isotopic. On $\partial V_i$, this is a
curve going twice in longitudinal direction, and we may choose
the longitude $\lambda_2$ on $\partial V_2$ such that this
becomes $2\lambda_2+\mu_2$, since the $V_i$ are cylinders over a disc
with bottom and top glued by a rotation of the disc through an angle~$\pi$.

In conclusion, the two circles $\mu_1'\cap T$
are isotopic in the thickened $2$-torus to
curves on $\partial V_2$ representing (in total)
the class $-2n(2\lambda_2+\mu_2)$. In addition, as we
see from Figure~\ref{figure:orbifold-heegaard}, there is
a copy of $-\mu_2$ in $\mu_1'$. Thus, the Heegaard splitting of the
lens space resulting from the Dehn filling of $\nu K$ has the
identification
\[ \mu_1\sim -\mu_2-2n(2\lambda_2+\mu_2)=-4n\lambda_2-(2n+1)\mu_2,\]
which is the description of $L(4n,-(2n+1))=L(4n,2n-1)$.
\end{proof}
\section{Explicit embeddings of Klein bottles}
Before we turn to lens spaces, for completeness we record a simple
construction of an embedding of the Klein bottle in $S^1\times S^2$.
One of the four explicit embeddings in $L(4n,2n\pm 1)$ will
be based on the same idea.

\begin{ex}
\label{ex:Moebius}
An embedding of the M\"obius band into a solid torus
$S^1\times D^2$ with boundary mapping to $2\lambda\pm\mu$
is given by
\[ \begin{array}{ccc}
([0,1]\times[-1,1]\bigr)/(1,r)\sim(0,-r) & \longrightarrow & S^1\times D^2\\
{[(t,r)]}                                  & \longmapsto     &
   \bigl(\rme^{2\pi\rmi t},r\rme^{\pm\pi\rmi t}\bigr);
\end{array}\]
see Figure~\ref{figure:mb-torus-plus}, which shows the embeddding
with boundary curve $2\lambda+\mu$.
The manifold $S^1\times S^2$ is obtained from two copies
of $S^1\times S^2$ via the identification $\mu_1\sim \mu_2$
and $\lambda_1\sim\lambda_2$. Thus, two M\"obius bands
embedded as described (with the same choice of sign)
will glue together to yield a Klein bottle
in $S^1\times S^2$.
\end{ex}

\begin{figure}[h]
\labellist
\small\hair 2pt
\endlabellist
\centering
\includegraphics[scale=0.6]{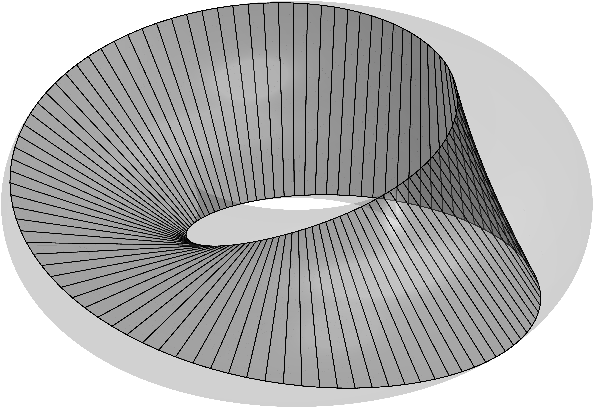}
  \caption{A M\"obius band in a solid torus with boundary
           $2\lambda+\mu$.}
  \label{figure:mb-torus-plus}
\end{figure}

Next we describe four explicit realisations of an embedding
$K\subset L(4n,2n\pm 1)$.
We mention in passing that all embeddings of $K$ in a given lens space $L$
are in fact isotopic, as was proved by Rubinstein~\cite{rubi78}.
A simple homological argument, using the fact that
$H_2(L)=0$, shows that an embedded
nonorientable surface of minimal genus must be incompressible,
and then \cite[Theorem~12]{rubi78} contains the isotopy statement.
\subsection{Embedding into a Seifert fibration}
\label{subsection:emb-Seifert}
In Section~\ref{subsubsection:seifert-rp2} we observed that an embedded Klein
bottle in a lens space leads to a Seifert fibration over $\RP^2(n)$,
the projective plane with one orbifold point of order~$n$, for
some $n\in\N$. This argument can be reversed.

From \cite{gela21} we know that the lens spaces $L(4n,2n\pm 1)$
admit a Seifert fibration over $\RP^2(n)$. Restricted to a
simple closed curve in $\RP^2(n)$ (disjoint from the orbifold point)
along which the orientation of the projective plane is reversed,
this Seifert bundle is the nontrivial $S^1$-bundle over~$S^1$,
which is a Klein bottle.
\subsection{Embedding two M\"obius bands}
According to Proposition~\ref{prop:lens-heegaard}, one
can find a Heegaard splitting of $L(4n,2n\pm 1)$ with
$r=n$ and $s=-(2n\mp 1)$ in~\eqref{eqn:matrix}.
Then
\[ 2\lambda_1+\mu_1\sim 2\bigl(-(2n\mp 1)\lambda_2+n\mu_2\bigr)+
\bigl(4n\lambda_2-(2n\pm 1)\mu_2\bigr)=\pm (2\lambda_2-\mu_2).\]
Hence, if we embed a M\"obius band in either Heegaard torus as in
Example~\ref{ex:Moebius}, but now with the opposite choice of signs,
they glue to a Klein bottle in the lens space.
\subsection{Embedding a handle decomposition}
The Klein bottle has a handle decomposition with one
$0$-handle, two twisted $1$-handles, and a single $2$-handle.
One may now try to place the $0$-handle as a meridional disc
in one Heegaard torus of $L(4n,2n\pm 1)$, the $1$-handles
on the boundary $2$-torus, and complete with a meridional
disc in the complementary Heegaard torus.

For the lens space $L(4,1)$, such an embedding has
been described in \cite[p.~460]{lrs15}. For $L(4n,2n\pm 1)$
with $n\geq 2$ one needs to modify the above idea. After attaching
the first $1$-handle, one has to push the handlebody slightly
into the Heegaard torus, keeping its boundary fixed on the splitting torus.
Only then one can attach the second $1$-handle. We shall presently
describe this in detail.

A similar description can be found in the unpublished note~\cite{iwak09}
by M.~Iwakura. Beware that on the arXiv this paper is listed under the
title `Geometrically incompressible non-orientable
closed surfaces in lens spaces', but the file that opens actually
carries the title `Non-orientable fundamental surfaces in lens spaces',
as listed in our references. There is a published paper
carrying that second title, under the joint authorship of Iwakura with
C.~Hayashi; this is \emph{not} the paper we are referring to.
The paper \cite{tsau92} seems to contain similar ideas, but not the
construction we are about to present.

We illustrate the splitting $2$-torus in the Heegaard decomposition
of $L(4n,2n\pm 1)$ by a square with opposite sides identified.
The horizontal direction corresponds to the meridional
direction~$\mu_2$; the vertical, to~$\lambda_2$. First we attach
a twisted $1$-handle to a meridional disc in
the Heegaard torus $M_2$ as shown in Figure~\ref{figure:first-handle}.
The horizontal line in the centre of the square represents
a meridian, i.e.\ the boundary of a meridional disc. The
grey band is the $1$-handle attached to this disc. (In spite of the
optical illusion, the segments where the $1$-handle intersects the
top and the bottom of the square really do match.) You may also
want to take a peek at Figure~\ref{figure:L41}.

\begin{figure}[h]
\labellist
\small\hair 2pt
\pinlabel $\mu_2$ [t] at 464 0
\pinlabel $\lambda_2$ [r] at 0 464
\endlabellist
\centering
\includegraphics[scale=0.3]{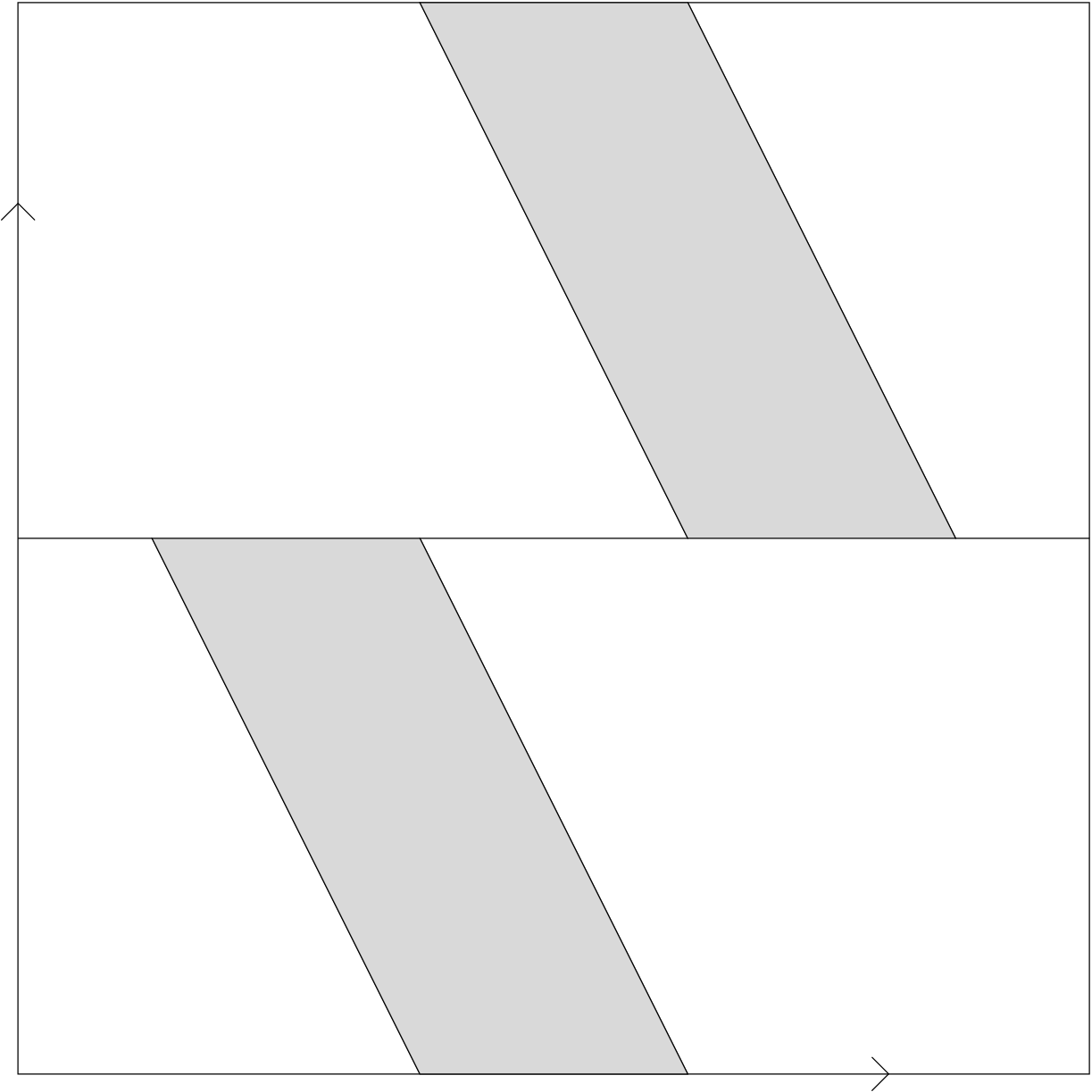}
  \caption{The first $1$-handle on the splitting $2$-torus.}
  \label{figure:first-handle}
\end{figure}

The second $1$-handle we draw as a single curve;
imagine this curve as being thickened into a band.
For $L(4,3)$ we attach the second $1$-handle as shown in
Figure~\ref{figure:L43}. The boundary of the resulting
$1$-handlebody is a curve that intersects $\mu_2$ in four points,
and $\lambda_2$ in three. Taking orientations into account, this
shows that the boundary curve lies in the class $4\lambda_2-3\mu_2$,
which by \eqref{eqn:matrix} becomes identified with $\mu_1$. So we
can complete the $1$-handlebody to a Klein bottle by attaching a meridional
disc of $M_1$ as a $2$-handle.

\begin{figure}[h]
\labellist
\small\hair 2pt
\pinlabel $\mu_2$ [t] at 464 0
\pinlabel $\lambda_2$ [r] at 0 464
\endlabellist
\centering
\includegraphics[scale=0.3]{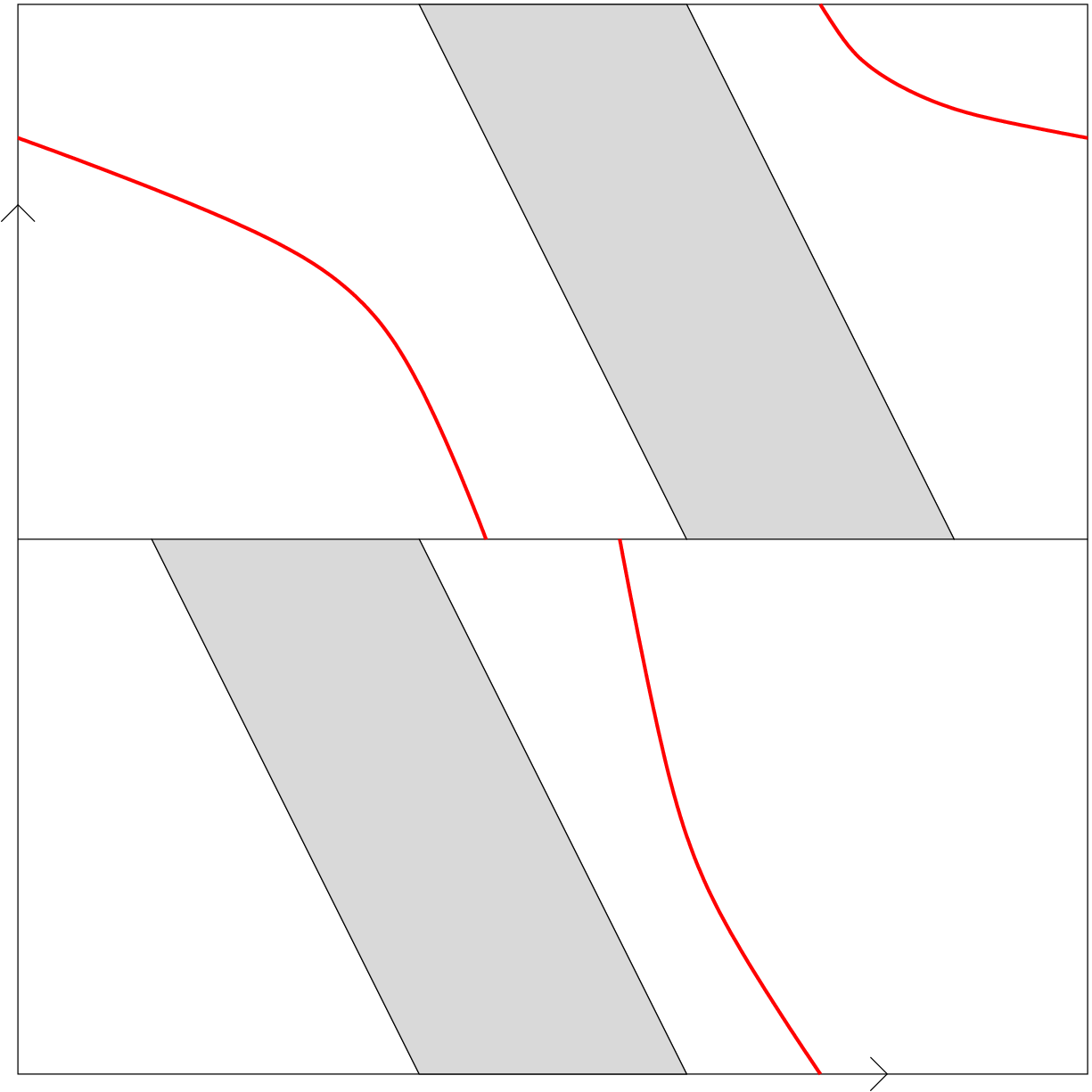}
  \caption{The second $1$-handle for $L(4,3)$.}
  \label{figure:L43}
\end{figure}

For $L(4n,2n+1)$ with $n\geq 2$ we need to modify this construction.
Again we start with the $1$-handlebody obtained by attaching a single
twisted $1$-handle to a meridional disc in~$M_2$.
Next we push the interior of this $1$-handlebody slightly into~$M_2$,
keeping the boundary curve on $\partial M_2$ fixed. This will
allow us to attach a second $1$-handle, lying entirely
in $\partial M_2$, as long as we stay away from the
boundary curve of the handlebody made up of the meridional disc
and the first $1$-handle, except at the ends of the second $1$-handle,
which we attach to the meridional disc.
For $L(4n,2n+1)$ we take the second $1$-handle as shown
in Figure~\ref{figure:second-handle+}, which illustrates
the general principle by the case $n=5$. Notice that the
second $1$-handle passes $n-1$ times over
the (original) first $1$-handle.

\begin{figure}[h]
\labellist
\small\hair 2pt
\pinlabel $\mu_2$ [t] at 464 0
\pinlabel $\lambda_2$ [r] at 0 455
\endlabellist
\centering
\includegraphics[scale=0.3]{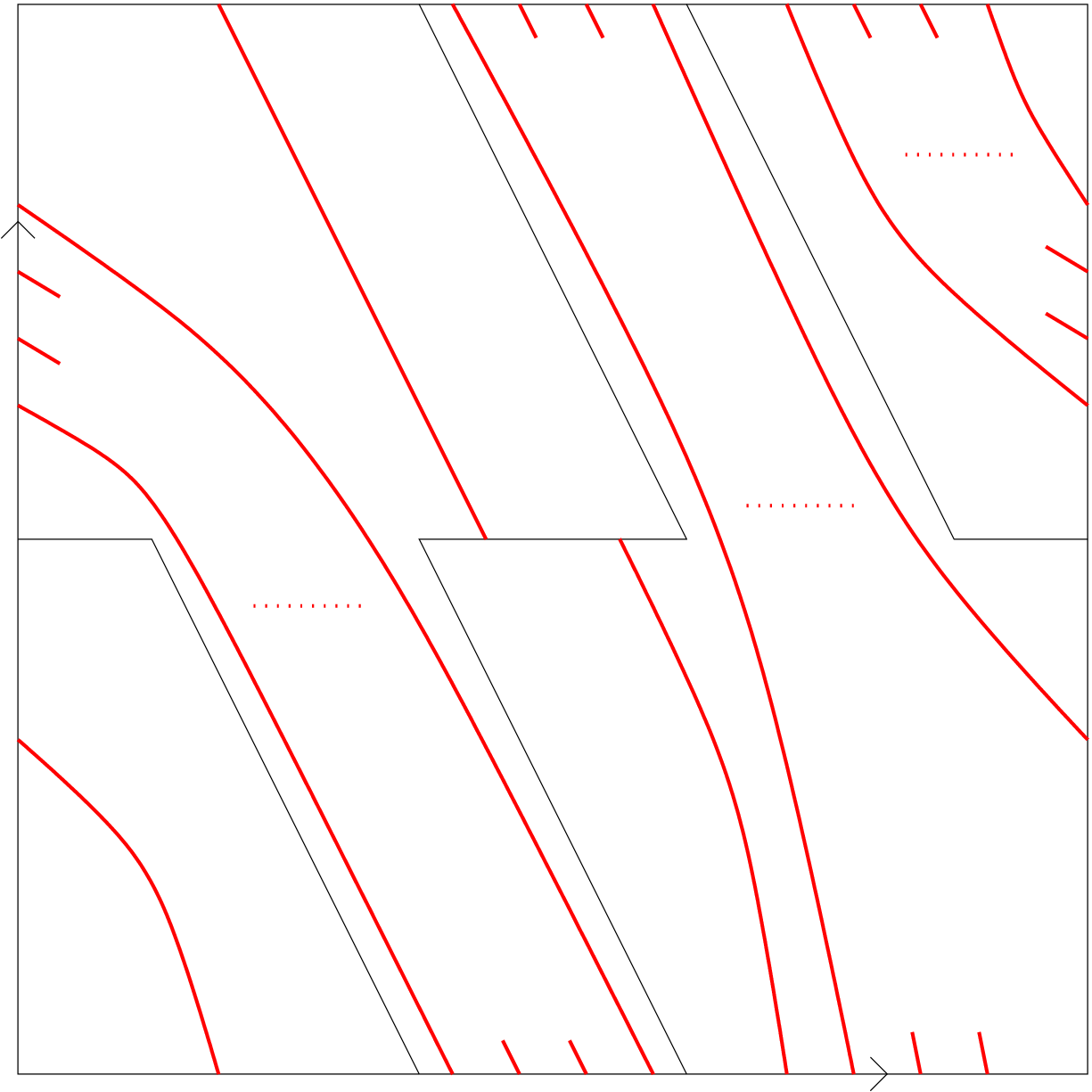}
  \caption{The second $1$-handle for $L(4n,2n+1)$, $n\geq 2$.}
  \label{figure:second-handle+}
\end{figure}

Now the boundary curve of the resulting $1$-handlebody intersects 
$\mu_2$ in $2+1+2(n-1)+1+2(n-1)=4n$ points, and $\lambda_2$ in
$2+1+2(n-1)=2n+1$ points, so it lies in the class
$4n\lambda_2-(2n+1)\mu_2$. Again, this is the
class of~$\mu_1$.

For $L(4n,2n-1)$, $n\in\N$, one can draw similar pictures,
or one appeals to the orientation-reversing
diffeomorphism from $L(4n,2n+1)$ to $L(4n,2n-1)$.
A $3$-dimensional visualisation of the $1$-handlebody in
$L(4,1)$ is shown in Figure~\ref{figure:L41}.

\begin{figure}[h]
\labellist
\small\hair 2pt
\endlabellist
\centering
\includegraphics[scale=0.35]{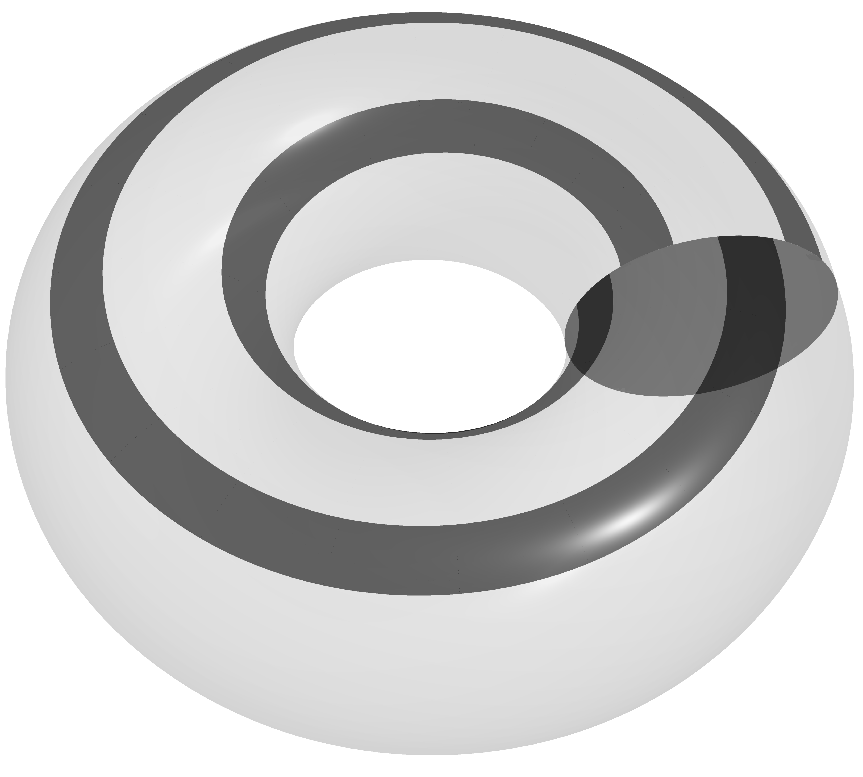}
  \caption{The $1$-handlebody of $K$ in $L(4,1)$.}
  \label{figure:L41}
\end{figure}

\subsection{Embedding into a lens model}
Before we specialise to $L(4n,2n\pm 1)$,
we want to describe a fundamental domain for the $\Z_p$-action
on the unit sphere $S^3\subset\C^2$ generated by $\sigma$
as in~\eqref{eqn:sigma}, producing the quotient $L(p,q)$.

In $S^3\subset\C^2$ we define the $2$-disc
\[ D:=\Bigl\{\bigl(\sqrt{1-r^2},r\rme^{\rmi\varphi}\bigr)\co
r\in[0,1],\, \varphi\in\R/2\pi\Z\Bigr\}.\]
The images of $D$ under the action of $\Z_p$ are discs that share
the boundary with~$D$, and which are
characterised by $\arg(z_1)$ being an integer
multiple of $2\pi/p$. As fundamental domain for the action we may take
the `lens'
\[ B:=\bigl\{(z_1,z_2)\in S^3\co \arg(z_1)\in[0,2\pi/p]\bigr\},\]
which is a homeomorphic copy of the $3$-ball with boundary
sphere $D\cup_{\partial D}\sigma(D)$. The lens space $L(p,q)$ is obtained
by identifying points $x\in D$ on the lower hemisphere
with $\sigma(x)$ on the upper hemisphere $\sigma(D)$. This
boundary identification amounts to a rotation of the lower hemisphere
through $2\pi q/p$ followed by vertical projection onto the upper hemisphere.

We now specialise to $p=4n$, $q=2n\pm 1$. As a model for the Klein bottle
we take
\[ K=\Bigl[0,\frac{\pi}{2n}\Bigr]\times[0,\pi]/\!\sim,\]
where $\sim$ denotes the boundary identification of the rectangle by
\[ (\varphi,0)\sim(\varphi,\pi)\;\;\;\text{and}\;\;\;
(0,\theta)\sim\Bigl(\frac{\pi}{2n},\pi-\theta\Bigr).\]
The embedding of the rectangle into $S^3$,
\[ \begin{array}{rccc}
\iota\co & \bigl[0,\frac{\pi}{2n}\bigr]\times[0,\pi] &
   \longrightarrow & S^3\\[1mm]
         & (\varphi,\theta)                          &
   \longmapsto     & \bigl(\sin\theta\,\rme^{\rmi\varphi},
                     \cos\theta\,\rme^{\pm\rmi\varphi}\bigr),
\end{array}\]
sends the rectangle to the fundamental domain~$B$, and its
boundary to~$\partial B$; see Figure~\ref{figure:lens-model},
which shows the situation for $L(8,3)$.

\begin{figure}[h]
\labellist
\small\hair 2pt
\endlabellist
\centering
\includegraphics[scale=0.6]{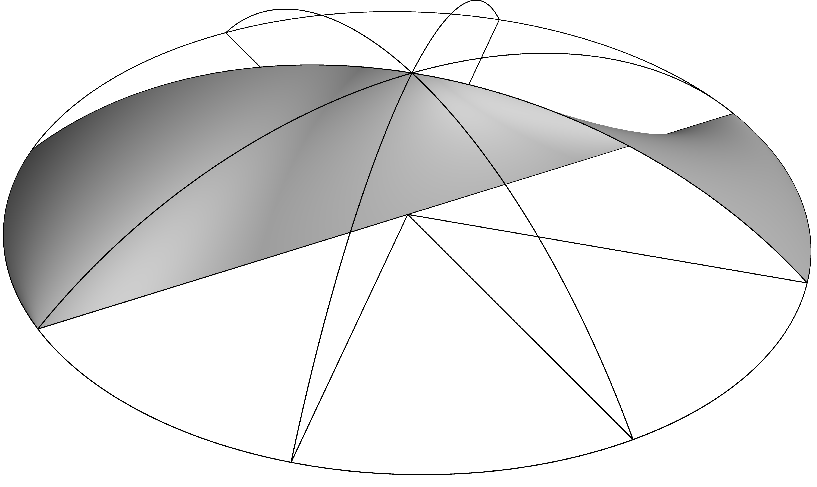}
  \caption{The embedding of $K$ in the lens model.}
  \label{figure:lens-model}
\end{figure}

This embedding descends to an embedding of $K$ into $L(4n,2n\pm 1)$, since
\[ \iota(\varphi,\pi)=\bigl(0,-\rme^{\pm\rmi\varphi}\bigr)=
\sigma^{2n}\bigl(0,\rme^{\pm\rmi\varphi}\bigr)=\sigma^{2n}\iota(\varphi,0)\]
and
\[ \iota\Bigl(\frac{\pi}{2n},\pi-\theta\Bigr)=
\bigl(\sin\theta\,\rme^{\pi\rmi/2n},-\cos\theta\,\rme^{\pm\pi\rmi/2n}\bigr)=
\sigma(\sin\theta,\cos\theta)=\sigma\iota(0,\theta).\]

\begin{rem}
In \cite{thie22} this embedding is related explicitly to the embedding
into the Seifert fibration and described in terms of the stereographic
projection of $S^3$ to~$\R^3$.
\end{rem}
\begin{ack}
This note is a companion paper to~\cite{gela21}. Both papers
(and the thesis project~\cite{thie22}) were prompted by
correspondence from Tye Lidman concerning an oversight in~\cite{gela18}.

H.G.\ is partially supported by the SFB/TRR 191
`Symplectic Structures in Geometry, Algebra and Dynamics',
funded by the DFG (Projektnummer 281071066 -- TRR 191).
\end{ack}

\end{document}